\newif\ifpictures
\picturestrue

\newif\ifcomment
\commenttrue

\documentclass[reqno, 11pt, a4paper, oneside]{amsart}

\def\MV{{\rm MV}}

\usepackage{latex_base}
\usepackage{macros}
\usepackage{breqn}
\usepackage{todonotes}

\usepackage[top=3cm, textwidth=16cm, vmarginratio=1:1]{geometry}

\author{Vadym Kurylenko}

\author{Benjamin Nill}

\dedicatory{Dedicated to the memory of Johanna Steinmeyer}

\title[Preserving Hodge vectors of lattice polytopes]{Preserving Hodge vectors of lattice polytopes}

\address{Faculty of Mathematics, Otto-von-Guericke-Universit\"at Magdeburg, Universit\"atsplatz 2, 39106 Magdeburg, Germany}

\email{vadym.kurylenko@ovgu.de, benjamin.nill@ovgu.de}

\begin{document}

\maketitle

 \begin{abstract}
 In Ehrhart theory, the local $h^*$-polynomial is a fundamental invariant of a lattice polytope. Its coefficient vector, the \emp{local $h^*$-vector}, is the coefficient vector of the top degree part of the Hodge-Deligne polynomial of the primitive cohomology with compact support of the associated generic hypersurface in the algebraic torus. In the literature on hypergeometric motives, the \emp{Hodge vector} of a lattice polytope is its local $h^*$-vector with leading and trailing zeroes removed. Of recent special interest are \emp{thin} polytopes whose Hodge vectors vanish.

Given lattice polytopes $P_1, \ldots, P_k$ contained in a $k$-dimensional subspace $U \subseteq \R^d$ and a $d$-dimensional lattice polytope $P \subset \R^d$, we compute the Hodge vector of the Cayley polytope $P_1 * \cdots * P_k * P$, and show that it equals the mixed volume of $P_1, \ldots, P_k$ times the Hodge vector of the projection of $P$ along $U$. This allows finding infinitely many high-dimensional lattice polytopes with the same Hodge vector that are not free joins. The proof relies on a closed formula for the Hodge-Deligne polynomial of generic complete intersections in the torus in terms of the bivariate/mixed $h^*$-polynomial. A special case of our construction is what we call \emp{Lawrence twists}: extending the Gale transform by centrally-symmetric pairs of vectors. As applications, we can produce many new thin polytopes answering a question by Borger, Kretschmer and the second author, and we provide an alternative explanation of the thinness of $B_k$-polytopes answering a question of Selyanin. 

 \end{abstract}

\section{Introduction and main result}

Let $P$ be a $d$-dimensional \emp{lattice polytope}, i.e., a $d$-dimensional polytope in $\R^d$ whose vertices are elements of the lattice $\Z^d$. Consider the associated affine hypersurface
 \[ Z_P ~ : ~ \left\{ x \in \left(\mathbb{C}^* \right)^d ~ : ~ \sum_{m \in P \cap \Z^d}  c_m x^m =0  \right\}, \] 
 where $c_m$ are complex numbers that are non-zero if $m$ is a vertex of $P$.  We consider situations only when $c_m$ are sufficiently generic, namely we require $Z_P$ to be nondegenerate with respect to $P$, also known as $P$-regular, see \cite[Definition 3.3]{batyrev_variations_1993}.
 
As will be recalled in \Cref{sec:hodge}, the Hodge-Deligne polynomial of the primitive cohomology of $Z_P$ can be written in the following form
\be E_{\rm prim}(Z_P; u,v) = \sum_{p+q \le d-1} h^*_{p,q} u^p v^q \label{hodge-diamond}\ee
 for nonnegative integers $h^*_{p,q}$ that satisfy Hodge duality $h^*_{p,q}=h^*_{q,p}$. These numbers were termed by Katz and Stapledon \cite[Remark~7.7]{katz_local_2016} as coefficients of the \emp{$h^*$-diamond} of $P$ and have purely combinatorial expressions \cite{danilov_newton_1987, batyrev_mirror_1996, borisov_string_2003, katz_local_2016} in terms of invariants of $P$. 

  Let us recall the main notion in classical Ehrhart theory. The \emp{Ehrhart series}, the generating series of the famous \emp{Ehrhart polynomial} of a lattice polytope, is the following rational function:
\[1+\sum_{k \geq 1} \mid k  P \cap \Z^d \mid t^k = \frac{h^*(P,t)}{(1-t)^{d+1}}. \] 
Its numerator $ h^*(P,t) = 1+\sum_{i=1}^d h_i^* t^i$ is called the \emp{$h^*$-polynomial} of $P$, and its coefficient vector $(1,h^*_1, \ldots, h^*_d)$ its \emp{$h^*$-vector}. A direct relation to algebraic geometry is given by the fact that for $i=1, \ldots, d$ the $i$th coefficient of the $h^*$-polynomial of $P$ equals the sum over the $i$th diagonal of the $h^*$-diamond (i.e., $h^*_i = \sum_{j=0}^{d-i}h^*_{i-1,j}$). Hence, the sum over all entries of the $h^*$-diamond plus one equals the normalized volume of $P$ (which is defined as $d!$ times the Euclidean volume of $P$).
 
 The palindromic top degree coefficient vector  $(h^*_{d-1,0}, h^*_{d-2,1}, \ldots, h^*_{0,d-1})$ of the\linebreak $h^*$-diamond is called the \emp{$\ell^*$-vector} or \emp{local $h^*$-vector} $\ell^*_P$ of $P$. This invariant was originally proposed by Stanley \cite{stanley_subdivisions_1992} in the context of polyhedral subdivisions and turned up independently in the combinatorial formulas for stringy Hodge numbers by Batyrev and Borisov \cite{batyrev_mirror_1996, batyrev_combinatorial_2008, nill_combinatorial_2012}. The study of local $h^*$-polynomials can also be referred to as \emp{local Ehrhart theory}. We refer to \cite{borger_thin_2023} for a comprehensive overview of the literature. In this paper, we emphasize its appearance in the theory of hypergeometric motives \cite{roberts_hypergeometric_2022}. From this viewpoint the question becomes relevant whether lattice polytopes of different dimensions can have the same $\ell^*$-vector up to a different number of outside zeroes. Let us make the notation used in this research community precise:

 \begin{dfn}
The \emp{Hodge vector} of $P$ is defined as the zero vector if $\ell^*_P=0$, and otherwise as the vector $\ell^*_P$ without beginning and trailing zeroes.
 \end{dfn}
 
For instance, for $\ell^*_P=(0,0,1,0,3,0,1,0,0)$ the Hodge vector of $P$ equals\linebreak $(1,0,3,0,1)$. In this area, the following question arises naturally \cite{roberts_hypergeometric_2022, GGFernando}. Suppose we have a $(d-1)$-dimensional hypersurface $Z_P$ that realizes some hypergeometric motive with Hodge vector of length $l<d$. Does there exist a variety $V$ of dimension $l-1$ realizing the same motive with the same Hodge vector? From this perspective, it is thus important to know what combinatorial constructions for lattice polytopes preserve their Hodge vectors in order to recognize whether such a dimension reduction might already be possible in the toric setting. 

\smallskip

In classical Ehrhart theory, the lattice pyramid construction provides a direct way to get from a $d$-dimensional lattice polytope a $(d+1)$-dimensional lattice polytope with the same $h^*$-polynomial. In fact, Batyrev proved in \cite{Batyrev-given} even a converse result, namely, fixing the $h^*$-polynomial forces lattice polytopes in high dimensions to be lattice pyramids. In particular, there are up to lattice pyramid constructions only \emp{finitely} many isomorphism classes of lattice polytopes with the same $h^*$-vector (see also \cite{Nill-given}).

As it turns out, in local Ehrhart theory the situation is quite different and much more open. Note that because of palindromicity, the dimensions of lattice polytopes with the same nonzero Hodge vector need to have the same parity (even/odd). In fact, the Hodge vector of any lattice pyramid just vanishes, see \cite[Lemma~4.5]{batyrev_combinatorial_2008}. One easy way though to get a $(d+2)$-dimensional lattice polytope with the same Hodge vector is to take the free join with a lattice interval of length $2$ (see \ref{subsec:Cayley}). This motivates the question, whether there are other possibilities for preserving the Hodge vector.

In this paper, we present a general construction that allows finding \emp{infinitely} many higher-dimensional lattice polytopes with the same Hodge vector that are not free joins. Before stating the main result, we provide some required terminology. 
Let us recall that the \emp{Cayley polytope} $P_1 * \cdots * P_k$ of lattice polytopes $P_1, \ldots, P_k$ in $\R^d$ is defined as
\[P_1 * \cdots * P_k := \conv(P_1 \times \{e_1\}, \ldots, P_k \times \{e_k\}) \subset \R^d \times \R^k.\]
We note that $\dim(P_1 * \cdots * P_k) = \dim(P_1 + \cdots + P_k) + k -1$.  For $k=d$ we define the \emp{ mixed volume} of $P_1,\ldots, P_d$ by 
\[ \MV(P_1, \ldots, P_d) = \sum_{\varnothing \neq I \subseteq [d]} (-1)^{d-\abs{I}} \vol_{d} \left( \sum_{i \in I} P_i \right), \] where $\vol_d$ is the standard Euclidean volume in $\mathbb{R}^d$. Finally,
for a rational subspace $U \subseteq \R^d$, let us consider the projection $\pi: \R^d \rightarrow \R^d/U$. Given a  lattice polytope $P \subset \R^d$,  we say that the lattice polytope $\pi(P)$ with respect to the lattice $\pi(\Z^d)$ is the \emp{projection of $P$ along $U$}. Our main result is as follows.   
\begin{thm}\label{thm:main}
Let $P_1,\ldots,P_k, P_{k+1} \subset \R^d$ be lattice polytopes, where the first $k$ polytopes are contained in a $k$-dimensional rational subspace $U$ and $\dim P_{k+1}=d \geq 1$. Let $V$ denote the mixed volume $\MV(P_1, \ldots, P_k)$ of $P_1, \ldots, P_k$. Then the Hodge vector of $P_1 * \cdots * P_{k+1}$ equals $V$ times the Hodge vector of the projection of $P_{k+1}$ along $U$. 
\end{thm}

The proof can be found in \Cref{sec:proof}. It follows the original setup of the paper \cite{danilov_newton_1987} by Danilov and Khovanskii and the approach used in \cite{di_rocco_discrete_2016}. As a crucial tool, we present in \cref{prop: mixed_h_relation} an explicit formula  for the Hodge-Deligne polynomial of generic affine complete intersections in the algebraic torus.

\smallskip

Let us describe how to apply \cref{thm:main} to construct infinitely many non-isomorphic high-dimensional lattice polytopes with the same Hodge vector. Let $P$ be a lattice polytope of dimension $d$ in $\R^d$. Choose a nonnegative integer $k$ and any full-dimensional lattice polytope $\tilde{P}$ in $\R^{d+k}$ that projects onto $P$ by projecting onto the first $d$ coordinates. Now, choose any set of $k$ lattice polytopes $P_1, \ldots, P_k$ in $\{0\} \times \R^k$ with mixed volume $1$ (for instance\footnote{Mixed volume one tuples of lattice polytopes were completely classified by Esterov and Gusev \cite{mvone}.}, just take $P_1=\cdots=P_k$ as the convex hull of the last $k$ standard basis vectors together with origin). Then by \cref{thm:main} the Cayley polytope $P_1 * \cdots * P_k * \tilde{P}$ of dimension $d+2k$ has the same Hodge vector as $P$. We call such a  construction a \emp{generalized Lawrence twist}, see \cref{Cayley-formula}. 

 We refer to the most special case of the construction of \cref{thm:main} as a \emp{Lawrence twist} and describe it closely in \Cref{sec:lawrence}. The name comes from its Gale-dual description that involves extending the Gale transform by centrally-symmetric pairs of vectors. 

\medskip

The paper is organized as follows: \Cref{sec:hodge} contains the basics on Hodge-Deligne polynomials, bivariate Ehrhart theory and Cayley polytopes, as well as the closed formula for the Hodge-Deligne polynomial for complete intersections in the torus. \Cref{sec:proof} contains the description of generalized Lawrence twists which implies a proof of \cref{Cayley-formula}. \Cref{sec:lawrence} discusses Lawrence twists. Applications, open questions and relations to other papers will be presented in \Cref{sec:applications}. 

\subsection*{Acknowledgments}
This work is funded by the Deutsche Forschungsgemeinschaft (DFG, German Research Foundation) – 539867500 as part of the research priority program Combinatorial Synergies.  The first author thanks Giulia Gugiatti for insightful conversations, Asem Abdelraouf and Fernando Rodriguez Villegas for useful discussions. We also thank Christian Haase for pointing us to \cite{Halit_Dur}. We are grateful to the anonymous referee for useful comments.

\section{Hodge-Deligne polynomials and bivariate Ehrhart theory} \label{sec:hodge}

This section recalls the main concepts needed for our results.

\subsection{Hodge-Deligne polynomials of affine hypersurfaces in the torus}
The cohomology (with compact support) of an algebraic variety $Z$ carries a mixed Hodge structure. This leads to the definition of its \emp{Hodge-Deligne polynomial}, also known as 
\emp{$E$-polynomial}:
 \[ E(Z; u, v) \coloneqq \sum_{p,q} \left(\sum_{m} (-1)^m  h^{p,q} (H^m_c(Z, \Q) \right) u^p v^q \in \Z[u,v].\]
Recall that $E(\C; u,v)=uv$ and $E(\C^*; u,v)=uv-1$. Hodge-Deligne polynomials have nice properties regarding products and stratifications:
\begin{itemize}
    \item For varieties $X$ and $Y$ we have 
    \[ E(X \times Y; u,v) = E(X;u,v) \cdot E(Y; u,v); \] 
    \item if $X = \sqcup_i X_i$ is a disjoint union of a finite set of locally closed subvarieties $X_i$, then
    \[ E(X; u,v) = \sum_i E(X_i; u,v). \] 
\end{itemize}
See \cite{danilov_newton_1987} for proofs and \cite{di_rocco_discrete_2016} for more examples.

\smallskip
Our particular interest lies in affine hypersurfaces $Z_P$ in algebraic tori. Let us follow \cite{danilov_newton_1987} and discuss further basic properties of the cohomology of $Z_P$. 
First of all, we can focus on studying $P$ of full dimension $d$. Otherwise, $P$ would define a hypersurface $Z'_P$ in a torus of lower dimension  $d'$ and we could write 
\begin{equation} E(Z_P; u,v) = (uv-1)^{d-d'} E(Z'_P; u,v). \label{low-dim}\end{equation}

Since $Z_P$ is an affine variety of dimension $d-1$, by Grothendieck vanishing theorem and Poincaré duality we know that  
\[ H_c^i(Z_P, \mathbb{Q}) =0, \quad \text{for } ~ i=0,\ldots, d-2. \] 
Moreover, the higher cohomology groups are also well understood due to the following proposition. 
\begin{prop}[\cite{danilov_newton_1987}, Proposition 3.9]
There exist   homomorphisms  
\[     \phi_i: ~H_c^i  (Z_P, \mathbb{C})    \rightarrow   H_c^{i+2} (  \left(\mathbb{C}^* \right)^d , \mathbb{C})  \]
that are isomorphisms for $i > d-1$ and surjective for $i=d-1$.    
\end{prop}
Therefore, we know the cohomology of $Z_P$ except for the middle one $H^{d-1}_c(Z_P, \mathbb{Q})$. This motivates the following definition. cf. \cite{batyrev_variations_1993}.

\begin{dfn}
    Define the \emp{primitive cohomology} of $Z_P$ to be the kernel of $\phi_i$
    \[ PH^i_c(Z_P) \coloneqq \ker \phi_i. \] 
\end{dfn}
 Since mixed Hodge structures form an abelian category, the cohomology $ PH^i_c(Z_P)$ carries a mixed Hodge structure. 
Multiplying with an overall sign, so that its coefficients are always nonnegative,  we define
 \[ E_{prim}(Z_P; u,v) \coloneqq 
  \sum_{p+q \leq d-1} h^{p,q} (PH^{d-1}_c(Z_P, \Q)) u^p v^q  . \] 
 
 Since we know that $ h^{p,p}(H^{n+p}((\mathbb{C}^*)^n, \mathbb{Q})) = \binom{n}{p}$, we can easily relate $E_{prim}$ to the Hodge-Deligne polynomial of $Z_P$:
\begin{equation}
\label{eprim_hyp} E_{prim}(Z_P; u,v) = (-1)^{d-1} E(Z_P; u,v) + \frac{(1- u v)^{d} -1 }{ u v }.  \end{equation}

\subsection{Bivariate Ehrhart theory}

Let $P \subset \R^d$ be a $d$-dimensional lattice polytope. It is possible to express the polynomials $E(Z_P;u,v)$ and thus $E_{prim}(Z_P;u,v)$ purely in combinatorial terms. This has been done by in \cite{batyrev_mirror_1996, borisov_string_2003}, in particular see the last equation of Proposition 5.5 in \cite{borisov_string_2003}. We refer to \cite{katz_local_2016,borger_thin_2023} for the details. 

In \cite[Def.~7.5]{katz_local_2016} Katz and Stapledon define a bivariate version of the $h^*$-polynomial $h^*(P;u,v)$ which they call \emp{mixed $h^*$-polynomial} of $P$. As the term `mixed' is typically used in convex geometry for tuples of convex bodies and, moreover, a mixed version of the $h^*$-polynomial had already been defined in \cite{mixedEhrhart}, we prefer the term \emp{bivariate $h^*$-polynomial}. Since we would like to avoid in this paper any unnecessary technicality, let us just present its elegant algebro-geometric explanation:

\be \label{eprim} h^*(P; u,v) = 1 + u~ v ~E_{prim}(Z_P;u,v).\ee

Note that this is a bivariate polynomial with nonnegative integer coefficients. Denoting the coefficients of the bivariate $h^*$-polynomial as elements of an \emp{$h^*$-diamond}, see \eqref{hodge-diamond}, we get
\[h^*(P;u,v) = 1+ \sum_{p+q \le d-1} h^*_{p,q} u^{p+1} v^{q+1}.\]

We remark that from the bivariate $h^*$-polynomial we easily recover the classical $h^*$-polynomial: $h^*(P;t,1) = h^*_P(t)$. 
The degree of $h^*_P(t)$ is an important measure of the complexity of a lattice polytope and is called the \emp{degree} $\deg(P)$ of $P$. Note that the degree is at most $d$.

We have $h^*(P;u,v) = h^*(P;v,u)$ and the degree of $h^*(P;u,v)$ is at most $d+1$. We define the \emp{bivariate local $h^*$-polynomial} or \emp{bivariate $\ell^*$-polynomial} as the top degree part of the bivariate $h^*$-polynomial:
\[\ell^*(P;u,v) := \sum_{p+q = d-1} h^*_{p,q} u^{p+1} v^{q+1}.\]

Again, specializing to the univariate case we get the \emp{local $h^*$-polynomial} or \emp{$\ell^*$-polynomial} $\ell^*_P(t) := \ell^*(P;t,1) = \sum_{i=1}^d \ell^*_i t^i$, where $\ell^*_i=h^*_{i-1,d-i}$ for $i=1, \ldots, d$. It is a palindromic polynomial with respect to symmetry along $\frac{d+1}{2}$. Its coefficient vector $(\ell^*_1, \ldots, \ell^*_d)$ is the \emp{local $h^*$-vector} or \emp{$\ell^*$-vector} of $P$ and equals the middle row of the $h^*$-diamond. We have that $\ell^*_1=\ell^*_d$ equals the number of interior lattice points of $P$. In the case of special interest where $\ell_P(t)=0$, so the Hodge vector vanishes, we call $P$ \emp{thin}. $P$ is called \emp{trivially thin} if the dimension is at least twice the degree. Note that by palindromicity of the $\ell^*$-vector, trivially thin polytopes are automatically thin. For more about the properties of the local $h^*$-polynomial and about thin polytopes, we refer to \cite{borger_thin_2023}.

We note that if $P$ is just a lattice point, then $h^*(P;u,v)=1$ and $\ell^*(P;u,v)=0$.

\smallskip

Rewriting \eqref{eprim_hyp} using the bivariate Ehrhart theory notation we get:
\begin{equation}  E(Z_P; u,v) = \frac{1}{uv} \left[(-1)^{d-1} h^*(P; u,v) + (uv- 1)^{d} \right].
\label{E-hyper}
\end{equation}
We will generalize this formula in \cref{prop: mixed_h_relation} below.

\begin{ex}
Let us illustrate the $h^*$-diamond of a $3$-dimensional lattice polytope $P$:

\begin{center}
\begin{tabular}{ccccc}
& & 0 & &\\
& 0 && 0 &\\
$h^*_{2,0}$ && $h^*_{1,1}$ && $h^*_{0,2}$\\
& $h^*_{1,0} $& & $h^*_{0,1}$&\\
&&$h^*_{0,0}$&&
\end{tabular}
\end{center}

These numbers have the following explicit combinatorial expressions (see \cite[Example~8.9]{katz_local_2016}), where $v^*$ denotes the number of vertices of $P$, $e^*$ and $f^*$ denote the number of lattice points in the relative interior of edges respectively facets of $P$, $i^*$ denotes the number of lattice points in the interior of $P$, and $(2i)^*$ the number of lattice points in the interior of $2P$:

\begin{center}
\begin{tabular}{ccccc}
& & 0 & &\\
& 0 && 0 &\\
$i^*$ && $(2i)^*\!\!-\!4i^*\!\!-\!f^*$ && $i^*$\\
& $f^*$ & & $f^*$&\\
&&$v^*\!\!+\!e^*\!\!-\!4$&&
\end{tabular}
\end{center}
\end{ex}

We remark that a lower bound theorem of Katz and Stapledon \cite[p.~184]{katz_local_2016} yields the non-trivial combinatorial identity $i^* \le (2i)^*-4i^*-f^*$, see also \cite[Cor.~4.6]{borger_thin_2023}.

\subsection{Isomorphisms, Cayley polytopes, free joins and lattice pyramids}
\label{subsec:Cayley}
Let us first recall that \emp{ isomorphisms} or \emp{ unimodular equivalences} of lattice polytopes $P,Q \subset \R^d$ are given by affine lattice isomorphisms: we write $P \cong Q$ if and only if there are $A \in \GL_d(\Z)$ and $b \in \Z^d$ with $Q = A \cdot P + b$.

Now, let us give the definition of Cayley polytopes which turn up quite naturally when considering complete intersections, cf. \cite{GKZ, batyrev_combinatorial_2008, di_rocco_discrete_2016}.
\begin{dfn}
Let $P_1, \ldots, P_k \subset \R^d$ be lattice polytopes. Then 
\[P_1 * \cdots * P_k := \conv(P_1 \times \{e_1\},  \ldots, P_k \times \{e_k\}) \subset \R^{d+k}\]
is called the \emp{Cayley polytope} of $P_1, \ldots, P_k$. Alternatively,
\be P_1 * \cdots * P_k \cong \conv(P_1 \times \{0\}, P_2 \times \{e_1\}, \ldots, P_k \times \{e_{k-1}\}) \subset \R^{d+k-1}.\label{cayley-second}\ee
\end{dfn}
Note that the dimension of $P_1 * \cdots * P_k$ equals $\dim(P_1 + \cdots + P_k)+k-1$.\\ Its degree is bounded by $\dim(P_1 + \cdots + P_k) \le d$, see \cite[Proposition~1.12]{batyrev_multiples_2007}.

\smallskip

In Ehrhart theory, the following special cases of Cayley polytopes play an important role.

\begin{dfn}
Let $P \subset \R^n$ and $Q \subset \R^m$ be lattice polytopes. Then 
\[P \circ Q := \conv(P \times \{0\} \times \{0\}, \{0\} \times Q \times \{1\}) \subset \R^{n+m+1}\]
is called the \emp{free join} of $P$ and $Q$. We call the free join of $P$ with a lattice point, a \emp{lattice pyramid} $\pyr(P)$ over $P$.
\end{dfn}

We remark that $\dim(P \circ Q) = \dim(P)+\dim(Q)+1$.

Recall that lattice simplices whose vertices form an affine lattice basis are called \emp{unimodular simplices}. The standard example of a $d$-dimensional unimodular simplex is $\conv(0,e_1, \ldots, e_d)$, which is denoted by $\Delta_d$. One can also describe unimodular simplices as successive lattice pyramids over a lattice point.

\smallskip
The bivariate $h^*$-polynomial is multiplicative with respect to free joins. 
\begin{prop} \label{prop: multiplicativity}
Let $P \subset \R^n$ and $Q \subset \R^m$ be lattice polytopes. Then 
\[h^*(P \circ Q;u,v)=h^*(P;u,v) \; h^*(Q;u,v).\]
In particular,
\[\ell^*(P \circ Q;u,v)=\ell^*(P;u,v) \; \ell^*(Q;u,v).\]
\end{prop}
This statement can be obtained from the  multiplicativity of Hodge-Deligne polynomials with respect to the Cartesian product. We give a short proof after  \cref{prop: mixed_h_relation}. 
Alternatively, a purely combinatorial proof can be obtained using the multiplicativity of the $\ell^*$-polynomial \cite[Remark~4.6(5)]{nill_gorenstein_2013} with respect to free joins and the multiplicativity of the toric $g$-polynomial with respect to products of posets.

In particular, we recover the well-known multiplicativity of the $h^*$-polynomial \cite[Lemma~1.3]{henk_lower_2009} as well as that of the $\ell^*$-polynomial. As an important case, note that 

\be h^*(\pyr(P);u,v)=h^*(P;u,v) \text{ and } \ell^*(\pyr(P);u,v)=0. \label{invformula} \ee

 \subsection{A closed formula for the Hodge-Deligne polynomial of an affine complete intersection in the torus}
 
  Consider a system of $k$ equations in $(\mathbb{C}^*)^d$ given by $f_1 = f_2 = \ldots = f_{k} =0$, where $f_i$  are generic Laurent polynomials with respect to their supports. Let $P_i$ be the corresponding Newton polytopes of $f_i$. Let $Y$ be the complete intersection in $(\mathbb{C}^*)^d$ defined as the vanishing locus of the above system. Danilov and Khovanskii \cite{danilov_newton_1987} gave a combinatorial algorithm for computing the Hodge-Deligne polynomial $E(Y;u,v)$. We use it here to obtain an explicit formula in terms of the bivariate $h^*$-polynomials.    For $\varnothing \not= I \subseteq [k]$ we define $P^I$ as the Cayley polytope of $(P_i)_{i \in I}$, and $d_I$ as the dimension of the Minkowski sum of $(P_i)_{i \in I}$. Hence, $\dim(P^I) = d_I+|I|-1$. Moreover, we define $P^{\varnothing} := \{0\}$ with dimension $d_\varnothing := 0$.

 \begin{prop} \label{prop: mixed_h_relation} In this situation, the following holds:
   \[ E (Y; u,v) =\frac{1}{(uv)^k} \sum_{I \subseteq[k]} (-1)^{d_I -\abs{I}} (uv-1)^{d-d_I} h^*(P^I; u,v).\]  
\end{prop}

Note that with our convention for $I=\varnothing$ this formula agrees for $k=1$ with \eqref{E-hyper}.

\begin{proof}

We remark that it is enough to prove this in the case $\dim(P_1 + \cdots + P_k)=d$. Otherwise, let $Y'$ be the complete intersection corresponding to the $d_{[k]}$-dimensional subspace $\aff(P_1 + \cdots + P_k)$ with respect to the lattice $\aff(P_1, \ldots, P_k) \cap \Z^d$. As in \eqref{low-dim} we see that 
\[E(Y;u,v)=(uv-1)^{d-d_{[k]}} E(Y';u,v).\]
If for $E(Y';u,v)$ the respective statement in \cref{prop: mixed_h_relation} holds, then we see that it also holds for $E(Y;u,v)$.

So, let us assume that $\dim(P_1 + \cdots + P_k)=d$. We define $P := P^{[k]} = P_1 * \cdots * P_k$, so $\dim P = d+k-1$. Here, we use the isomorphic embedding of $P$ in $\R^{d+k-1}$ as in \eqref{cayley-second}. Let $Z_P \subseteq (\mathbb{C}^*)^{d+k-1}$ be a generic hypersurface with Newton polytope $P$. It can be defined  by the vanishing of 
\[ Z_P ~:~ f_1 + y_2 f_2 + \ldots + y_{k} f_{k} =0.  \] 

Instead of $Z_P$ let us now consider the hypersurface  $\tilde{Z} \subseteq (\mathbb{C}^*)^{d} \times \mathbb{C}^k$
\[ \tilde{Z} ~:~ 1+ y_1 f_1 + y_2 f_2 + \ldots + y_{k} f_{k} =0. \] 
Danilov and Khovanskii \cite[Section 6]{danilov_newton_1987} gave the following simple formula 
 \begin{equation}
     E(\Tilde{Z}; u,v) = (u v)^{k-1} \left( \left( uv -1 \right)^{d} - E (Y; u,v) \right). \label{E-tilde}
     \end{equation}

Note  that $\tilde{Z}$  is not a hypersurface in an algebraic torus. However, there is a stratification of $\tilde{Z}$ given by affine hypersurfaces in tori that was described in \cite{di_rocco_discrete_2016}. For $\varnothing \not= I \subseteq [k]$ define 
\[ {Z}_I = \Tilde{Z} \cap \{ y_j \neq 0 ~ :~ j \in I\} \cap \{ y_j =0 ~: ~j \notin I \}. \] 
Each of these is now an affine hypersurface in $(\mathbb{C}^*)^{d+|I|}$. Let us denote by $Q^I$ the corresponding Newton polytopes. Note that each $Q^I$ equals the lattice pyramid over the Cayley polytope $P^I$, so \[ \dim Q^I = d_I+|I|. \] 

From the stratification of $\Tilde{Z}$ and \eqref{E-tilde} we get
\[ (u v)^{k-1} \left( \left( u v -1 \right)^{d} - E(Y; u,v) \right) = \sum_{\varnothing\not= I \subseteq [k]} E({Z}_I ; u,v).   \] 

Plugging in the combinatorial formula \eqref{E-hyper} for $E({Z}_I ; u,v)$ and applying \eqref{low-dim} to take into account the difference $d+\abs{I}-(d_I+\abs{I})=d-d_I$ of the dimensions of $Q^I$ and the ambient space $\R^{d+\abs{I}}$ we get:

\[ (uv)^{k-1} \left( \left( u v -1 \right)^{d} - E(Y; u,v) \right) =\] 
\[\sum_{\varnothing\not=I \subseteq [k]} \frac{(uv-1)^{d-d_I}}{uv} \left[(-1)^{d_I+\abs{I}-1} h^*(Q^I;u,v)+(uv-1)^{d_I+\abs{I}}\right].\]

By \eqref{invformula} we know that the bivariate $h^*$-polynomial is invariant under lattice pyramids, so we get for $E (Y; u,v)$ the following expression:
\begin{equation}(uv-1)^d-\frac{1}{(uv)^k} \sum_{\varnothing\not=I \subseteq [k]}  \bigg[ (-1)^{d_I+\abs{I}-1}(uv-1)^{d-d_I} h^*(P^I;u,v)+ (uv-1)^{d+\abs{I}} \bigg].\label{newexpr}\end{equation}
It is direct to see that
\[\sum_{I \subseteq [k]} (uv-1)^{|I|} = (uv)^k,\]
so 
\[\sum_{\varnothing\not=I \subseteq [k]} (uv-1)^{d+|I|} = (uv-1)^d((uv)^k-1).\]
Plugging this into \eqref{newexpr} we get the desired formula for $E(Y;u,v)$.
\end{proof}

Note that for $k>d$, the left-hand side of \cref{prop: mixed_h_relation} vanishes, and thus it gives us a nontrivial combinatorial identity for the right-hand side.  

\begin{proof}[Proof of \cref{prop: multiplicativity}]
Let $Z_P$ and $Z_Q$ be the hypersurfaces in $(\C^*)^{n+m}$ corresponding to $P\times \{0\}$ and $\{0\} \times Q$ respectively. Then the complete intersection $Y$ from the proof of \cref{prop: mixed_h_relation} is exactly the product $Z_P \times Z_Q$. Since we have $E(Z_P \times Z_Q)=E(Z_P) \cdot E(Z_Q)$, then using the equations \eqref{eprim_hyp} and \eqref{eprim} together with \cref{prop: mixed_h_relation} one arrives at the desired multiplicativity formula. 
\end{proof}

\begin{ex} \label{prop:relation_to_MV}
Suppose that $k=d$, in this case the Hodge-Deligne polynomial $E(Y;u,v)$ is just the number of points of $Y$, which equals  the mixed volume of $P_1,\ldots, P_d$ by the Bernstein–Khovanskii–Kushnirenko (BKK) theorem \cite{bernstein_number_1975,kouchnirenko_polyedres_1976}. Suppose also that  for each $\varnothing \neq I \subsetneq [d]$ we have $d_I - \abs{I} >0$, for example, if all $P_i$ are full dimensional, then the Hodge vector of the Cayley sum $P^{[d]}$ is 
\[ l^*(P_1 * \ldots *  P_d, t) = \left(MV(P_1,\ldots,P_d)-1\right)\cdot t^{d}.  \]
\label{easy-cayley}
\end{ex}
 
 \section{Generalized Lawrence twists} \label{sec:proof}

Here is our main result. Theorem~\ref{thm:main} is a special case of it (see situation (1)).

\begin{thm} \label{thm:main-local} \label{Cayley-formula}
Let $P_1,\ldots,P_k, P_{k+1} \subset \R^d$ be lattice polytopes 
such that \[\dim(P_1 + \cdots + P_k) \le k \text{ and } \dim(P_1 + \cdots + P_{k+1})=d \geq 2.\]
Let $U$ be any $k$-dimensional rational subspace of $\R^d$ containing $P_1, \ldots, P_k$ if $k \le d$, and set $U := \R^d$ otherwise.
We denote by $V$ the mixed volume $\MV(P_1, \ldots, P_k)$ of $P_1, \ldots, P_k$ if $k \le d$, and set $V := 0$ otherwise.
 Then the bivariate $h^*$-polynomial of the Cayley polytope $P_1 * \cdots *  P_{k+1}$ equals
\be V (uv)^k h^*(\proj_U(P_{k+1}); u,v) + \sum_{I \subsetneq[k+1]: k+1 \in I} (-1)^{k -\abs{I}} (1-uv)^{d-d_I} h^*(P^I; u,v).
\label{main-eq}
\ee
In particular, we have
\be \ell^*(P_1 * \cdots *  P_{k+1};u,v) = V (uv)^k \ell^*(\proj_U(P_{k+1});u,v)
\label{main-local}\ee
in each of the following two situations: \begin{enumerate}
    \item For each $I \subsetneq[k+1]$ with $k+1 \in I$ we have 
    \begin{itemize}
        \item $|I|+(d-d_I) \le k$ or
        \item $|I|+(d-d_I)=k+1$ and $P^I$ is thin. 
    \end{itemize}
    For instance, this holds if $\dim(P_{k+1})=d$.
    \item $k\ge d$, which implies that $P_1 * \cdots *  P_{k+1}$ is trivially thin.
    \end{enumerate}
\end{thm}

\begin{proof}
Let us consider the main case $k\leq d$ first. 
Let $f_1, \ldots, f_{k+1} \in \C[x^\pm_1, \ldots, x^\pm_d]$ be generic polynomials with Newton polytopes $P_1,\ldots, P_{k+1}$. As $U$ is a rational subspace, we can choose an appropriate lattice transformation of $\Z^d$, so we can assume that $U$ is simply the subspace generated by the first $k$ standard basis vectors and $f_1, \ldots, f_k \in \C[x^\pm_1, \ldots, x^\pm_k]$.

We denote the solution set of $f_1, \ldots, f_k$ in $(\C^*)^k$ by $\hat{Y}$. Note that $\hat{Y}$ is a finite set. It follows from the BKK-theorem that its size equals the mixed volume $V := \MV(P_1, \ldots, P_k)$. In particular, applying \cref{prop: mixed_h_relation} to $P_1, \ldots, P_k$ and using the notation of \cref{prop: mixed_h_relation} yields
   \begin{equation}V =\frac{1}{(uv)^k} \sum_{I \subseteq[k]} (-1)^{d_I -\abs{I}} (uv-1)^{k-d_I} h^*(P^I; u,v).\label{mv}\end{equation}   
Applying \cref{prop: mixed_h_relation} to $P_1, \ldots, P_{k+1}$ gives
   \[ E (Y; u,v) =\frac{1}{(uv)^{k+1}} \sum_{I \subseteq[k+1]} (-1)^{d_I -\abs{I}} (uv-1)^{d-d_I} h^*(P^I; u,v).\]
By taking \eqref{mv} into account, we get that $E(Y; u,v)$ equals
\begin{equation}
\frac{1}{(uv)^{k+1}} \left((uv-1)^{d-k} (uv)^k V + \sum_{I \subseteq[k+1]: k+1 \in I} (-1)^{d -\abs{I}} (1-uv)^{d-d_I} h^*(P^I; u,v)\right).
\label{E-eq1}
\end{equation}

On the other hand, $Y$ is stratified in $\{x'\} \times Z_{x'}$ for $x' \in \hat{Y}$, where $Z_{x'}$ is the hypersurface in $(\C^*)^{d-k}$ defined by the polynomial $g_{x'}(z):=f_{k+1}(x',z) \in \C[z^\pm_1, \ldots, z^\pm_{d-k}]$. Note that each $g_{x'}$ is a generic\footnote{
These Laurent polynomials are nondegenerate with respect to  $\text{Newt}(g_{x'}(z))$ for each $x'$, provided that the coefficients of $f_{k+1}$ are chosen appropriately. Specifically, let $f_{k+1}= \sum_{m \in P_{k+1} \cap \Z^d} c_m x^m$. For $f_{k+1}$ to be nondegenerate with respect to $P_{k+1}$, the coefficient vector $(c_m)_{m \in P_{k+1}}$ must lie in the complement of a finite collection of hypersurfaces in the coefficient space $\mathbb{C}^{|P_{k+1} \cap \mathbb{Z}^d|}$. Similarly, for each $x'$, the requirement that $g_{x'}(z)$ be nondegenerate with respect to its own Newton polytope excludes another finite set of hypersurfaces. Consequently, there exists a Zariski open subset of the coefficient space in which $f_{k+1}$ and all $g_{x'}(z)$ are simultaneously nondegenerate. } polynomial with respect to its Newton polytope which equals the projection $\proj_U(P_{k+1})$ of $P_{k+1}$ along $U$.  We may assume $k<d$ as otherwise (in case (2)), $Z_{x'}$ and hence $Y$ is empty. Note that $\proj_U(P_{k+1})$ is a full-dimensional lattice polytope in $\R^{d-k}$. By \eqref{E-hyper} each $Z_{x'}$ has the same Hodge-Deligne polynomial and as $|\hat{Y}|=V$ we get
\begin{equation}
E (Y; u,v) =  \frac{V}{uv} \left[(-1)^{d-k-1} h^*(\proj_U(P_{k+1}); u,v) + (uv- 1)^{d-k} \right].
\label{E-eq2}
\end{equation}

Equating \eqref{E-eq1} and \eqref{E-eq2} and isolating $h^*(P_1 * \cdots *  P_{k+1};u,v)$ we get our desired formula after some cancellation. 

For the case when $k>d$, the equations  \eqref{mv} and  \eqref{E-eq1} that one derives 
from \cref{prop: mixed_h_relation} still hold with $V=0$. Thus, from these two equations one arrives  at 
\[ h^*(P_1 * \ldots * P_{k+1};u,v) =\sum_{I \subsetneq[k+1]: k+1 \in I} (-1)^{k -\abs{I}} (1-uv)^{d-d_I} h^*(P^I; u,v). \]

For the additional statements we recall that the Cayley polytopes $P^I$ have dimensions $d_I+|I|-1$, $\proj_U(P_{k+1})$ has dimension $d-k$, and the local $h^*$-polynomial of an $n$-dimensional lattice polytope equals the top degree part of its bivariate $h^*$-polynomial which has degree $n+1$. So, we need to consider the degree $d+k+1$ part in \eqref{main-eq}.
The first term in \eqref{main-eq} has exactly $d+k+1$ as maximal degree and $ V (u v)^k \ell^*(\proj_U(P_{k+1};u,v)$ as the top degree part. Therefore, to arrive at \eqref{main-local} we need that for each $I$ the parts of $h^*(P^I;u,v)$  with degrees $d+k+1-2j$ for $j=0,\ldots, d-d_I$ vanish. 
The easiest possible case is when the maximal degree of $h^*(P^I;u,v)$, i.e., $d_I+\abs{I}$, is smaller than $2d_I-d+k+1$, i.e., $|I|+(d-d_I) \le k$. 
Another notable situation, is when the maximal degree of $h^*(P^I;u,v)$ is exactly  $2d_I-d+k+1$, but the part of $h^*(P^I;u,v)$ of this degree vanishes, i.e. $P^I$ is thin.  

In the situation (2) note that the degree of the Cayley polytope $P_1 * \cdots *  P_{k+1}$ is at most $d$ while its dimension is $d+k\ge 2d$, so it is trivially thin. The projection $\proj_U(P_{k+1})$ is just a point, so its Hodge vector is also $0$.

\end{proof}

\begin{dfn}
Whenever \eqref{main-local} in \cref{Cayley-formula} holds, we call $P_1 * \cdots * P_{k+1}$ a \emp{generalized Lawrence twist} of $\proj_U(P_{k+1})$.
\end{dfn}

\begin{rem}
We leave it as an exercise to check that if $\hat{P}$ is a generalized Lawrence twist of $P$, and $\hat{\hat{P}}$ is a generalized Lawrence twist of $\hat{P}$, then $\hat{\hat{P}}$ is a generalized Lawrence twist of $P$.\label{comp}
\end{rem}

Let us also remark that \cref{thm:main} is reminiscent of the following well-known projection formula for the mixed volume, see \cite[Theorem~5.3.1]{Sch93} which one can also prove quite directly using the BKK-theorem. 
\begin{lem}
    Let $P_1, \ldots, P_d$ be lattice polytopes in $\R^d$. If $P_1, \ldots, P_k$ (for $1 \le k\le d$) are contained in a $k$-dimensional rational subspace $U$ of $\R^d$, then 
    \[\MV(P_1, \ldots, P_d) = \MV(P_1, \ldots, P_k) \cdot \MV(\proj_U(P_{k+1}), \ldots, \proj_U(P_d)).\]
\end{lem}

\section{Lawrence twists}

\label{sec:lawrence}

\subsection{The original motivation}

Let us explain where the motivation for the notion of generalized Lawrence twists comes from. For this, let us go back to the situation of circuits. Recall that a tuple of integers $\gamma = (\gamma_1, \ldots, \gamma_{d+2})$ with $\gcd(\gamma)=1$ defines uniquely a lattice polytope  in $\R^d$ with $d+2$ vertices, a \emp{circuit}, with circuit relation $\gamma$. One way to describe such a circuit is to project the unimodular simplex $\Delta_{d+1}$ along the $1$-dimensional subspace $\R (\gamma_1 0 + \gamma_2 e_1 + \cdots + \gamma_{d+2} e_{d+1})$. 

Moreover, we can associate to it a pair of tuples of rational numbers $(\a_1, \ldots, \a_K)$ and $(\b_1, \ldots, \b_K)$ for some positive integer $K$ defined by
\[ \frac{\prod_{i:\gamma_i<0} (T^{-\gamma_i}-1)}{\prod_{i:\gamma_i>0} (T^{\gamma_i}-1)}  = \frac{\prod_{i=1}^K (T-e^{2 \pi i \a_i})}{\prod_{i=1}^K (T-e^{2 \pi i \b_i})}.    \] 
Recall that from this data Corti, Golyshev and Fedorov \cite{corti_hypergeometric_2011, fedorov_variations_2018} gave a formula for the $r$-th coefficient of the local $h^*$-polynomial of the circuit (for $1 \le r \le d$):
\[ \ell_r^* = \# \left\{ j \mid \# \left\{  \a_i \mid \a_i \leq \b_j \right\} - j + m_-  = r   \right\},  \]
where $m_- = \abs{ \set{ i \suchthat \g_i<0}}$. 

Now note the following. If we extend the tuple $(\gamma_1, \ldots, \gamma_{d+2})$ by adjoining to it a tuple of integers of the form $(y_1,-y_1, y_2,-y_2, \ldots, y_k, -y_k)$, then this does not affect the definition of alphas and betas.  In the above formula for $\ell^*_n$ it affects only $m_-$, therefore, the local $h^*$-polynomial of the  circuit extended this way is the local $h^*$-polynomial of the initial circuit multiplied with $t^k$. In other words, they have the same Hodge vector. As it will turn out, this phenomenon can be explained by viewing this construction as a special instance of a generalized Lawrence twist. 

\subsection{Lawrence twists via Gale duality}

Let us very quickly recall the basics of Gale duality. Let $A \subseteq \mathbb{Z}^d$ be an integral $d$-dimensional point configuration consisting of $n>d+1$ points. We assume that $A$ is \emp{spanning}, i.e., any point in $\Z^d$ is an affine integer combination of elements of $A$. Let $\bar{A}$ denote the corresponding homogeneous point configuration in $\Z^{d+1}$. There exists an integer $(n-d-1) \times n$-matrix $G$ such that $\bar{A} \cdot G^T=0$ and the columns of $G$ form a vector configuration that spans $\Z^{n-d-1}$. We note that the sum over the columns of $G$ is zero. We call $G$ a \emp{Gale transform} of $A$. For instance, in the case of the circuit above, we have $n=d+2$ and $G=\gamma$.
Conversely, every such integer $(n-d-1) \times n$-matrix $G$ whose columns span $\Z^{n-d-1}$ and sum up to zero arises in this way. 
Under these assumptions, we get a well-defined correspondence between unimodular equivalence classes of point configurations $A$ and unimodular equivalence classes of vector configurations $G$.

\smallskip

\begin{ex}
    Let $P$ be a $d$-dimensional lattice polytope $P \subset \R^d$. Let us assume that $P$ is \emp{spanning}, i.e., $P \cap \Z^d$ is a spanning point configuration, and $|P \cap \Z^d| > d+1$. Then we call the Gale transform of $P \cap \Z^d$ the \emp{Gale transform} of $P$. It is a challenging open problem to express the Hodge vector of $P$ using just its Gale transform in a way similar to the formula of Corti, Golyshev and Fedorov for circuits. 
\end{ex}

\begin{dfn}
 Let $A$ be a spanning point configuration satisfying $|A|\, > d+1$. We denote its Gale transform by $G_A$. Let $\tilde{A} \subset \Z^{d+2k}$ be the spanning point configuration with associated Gale transform $G_A \sqcup S_k$, where $S_k$ is a centrally symmetric configuration consisting of $ 2 k $ non-zero integer vectors.

  Then we call the convex hull of $
  \tilde{A}$ a \emp{Lawrence twist} of the convex hull of $A$.  
  It is a $(d+2k)$-dimensional spanning lattice polytope.

  In general, a spanning lattice polytope $\tilde{P}$ is a Lawrence twist of a spanning lattice polytope $P$ if one can choose $A$ and $\tilde{A}$ as described such that $\tilde{P} = \conv(\tilde{A})$ and $P = \conv(A)$. 
  \label{def-lawrence}
\end{dfn}

Now, our choice of terminology for a Lawrence twist should have become more clear. It has been motivated by the famous class of \emp{Lawrence polytopes} \cite{bayer_lawrence_1990}: those polytopes whose Gale transform is given by a centrally symmetric configuration of vectors. A Lawrence twist can also be seen as a variant of a \emp{Lawrence lift}, where one adds the negative of an existing vector in the Gale transform \cite[Section~5.5]{de_loera_triangulations_2010}. 

\subsection{Lawrence twists are generalized Lawrence twists}

Let us give the interpretation of Lawrence twists in the primal space. It suffices to consider the case $k=1$, where we extend the Gale transform by one pair of centrally-symmetric vectors.

\begin{prop}
Let $P \subset \R^d$ be a spanning lattice polytope, $a_1, \ldots, a_n$ a spanning point configuration consisting of some of the lattice points of $P$ including the vertices of $P$, and let $n > d+1$. Let $P_1 := \conv(0,e_{d+1})$, and $P_2$ be the convex hull of $(a_1, c_1), \ldots, (a_n, c_n) \in \Z^{d+1}$, where $c_1, \ldots, c_n \in \Z$, with $(c_1, \ldots, c_n) \not= (0,\ldots,0)$. Then $P_1*P_2$ is a Lawrence twist of $P$ with $k=1$. Moreover, any Lawrence twist of $P$ with $k=1$ is given in this way.
\label{prop:lawrence}
\end{prop}

\begin{proof}
Denote the columns of $G_A$ by $b_1, \ldots, b_n \in \Z^{n-d-1}$. Choose an integer non-zero vector $v \in \Z^{n - d - 1}$. By our spanning assumption, 
there exist $c_1,\ldots,c_n \in \Z$ such that $c_1 b_1 + \cdots + c_n b_n = v$. 
Now, it can be directly checked that the point configuration associated to $G_A\cup \{v, -v\}$ is given by the columns of  
    \[ \tilde{A}  = \begin{pmatrix}
         a_{11} & \ldots& a_{n1} & 0 & 0 \\
         a_{12}& \ldots &a_{n2} & 0 & 0 \\
         \vdots &\ldots &\vdots & 0 & 0  \\ 
         a_{1d} &\ldots& a_{nd} & 0 & 0 \\
         c_1 & \ldots & c_n & 0 & 1 \\ 
         0 & \ldots & 0 & 1 & 1 \\ 
     \end{pmatrix}. \] 
Their convex hull is precisely the Cayley polytope $P_1 * P_2$. From this, the statements follow.
\end{proof}

We observe from \cref{Cayley-formula}(1) that this is just a special case of a generalized Lawrence twist. Hence, \cref{comp} implies directly a result proven by the first author in his thesis \cite{vadym_thesis}.

\begin{cor} \label{lawrence_twist}
Lawrence twists preserve Hodge vectors.
\end{cor}

\subsection{A direct relation between \texorpdfstring{$h^*$}{h*}-polynomials of lattice projections and Cayley polytopes}

\cref{Cayley-formula}(1) gives an explicit and easy description of the bivariate $h^*$-polynomial of a Lawrence twist (with $k=1$) just using the bivariate $h^*$-polynomials of the projection and the polytope itself. Turning this around, this could also be seen as an explicit formula for the $h^*$-polynomial of a projection.

\begin{cor} Let $P \subseteq \R^d$ be a $d$-dimensional lattice polytope and let $I \subset \R^d$ be a lattice interval of normalized volume $1$. Then we have
\[ h^*( P * I; u, v)= u v \cdot h^*( \proj_I P; u,v) + h^*(P;u,v). \] 
In particular, 
\[ h^*( P * I, t) = t \cdot h^*( \proj_I P, t) + h^*(P,t).\] 
This implies
\[\deg(P * I) \ge \deg(\proj_I P) + 1.\]
\end{cor}

\cref{prop:lawrence} implies the following observation.

\begin{cor}
Lawrence twists with $S_k$ increase the degree at least by $k$.
\label{cor-degree}
\end{cor}

\begin{rem}
In the bachelor thesis of Halit Dur \cite{Halit_Dur} the above formula was discussed in the special case when $\proj_I P$ is a face of $P$ and $I$ is the normal vector of this face. Moreover, it was noted that this way one can construct a sequence of polytopes $P_0,P_1,P_2, \ldots$ whose $h^*$-polynomials behave similarly to the Fibonacci polynomials, albeit with the non-standard starting values $h^*(P_0)=h^*(P_1)=1$. 
\end{rem}

\section{Applications}
\label{sec:applications}

\subsection{Infinitely many non-free-joins lattice polytopes with same Hodge vector}

This follows from the following result, whose proof
we leave to the reader. A detailed argument using Gale duality can be found in the thesis of the first author \cite[Lemma 13]{vadym_thesis}.

\begin{lem}
 Let $P$ be a spanning lattice polytope of dimension $d$. For each $k \in \Z_{\ge 1}$ there are infinitely many non-isomorphic Lawrence twists of $P$ of dimension $d+2k$ that are not free joins. \label{lem-lawrence}
\end{lem}

\begin{rem}
Note that for the proof one really has to work on the level of point configurations and not only on the level of polytopes. For instance, if $P$ is a lattice pyramid with apex $v$ and one takes in \cref{def-lawrence} as $A$ simply $P \cap \Z^d$, then its Lawrence twist is also a lattice pyramid, so a free join. One needs to take as $A$ the point configuration $P \cap \Z^d$ with $v$ as a double point in order for the Lawrence twist to be not necessarily a lattice pyramid anymore. 
\end{rem}

\subsection{Many more thin polytopes}

Recall that lattice polytopes with vanishing Hodge vectors are called thin. So far, two ways were known to get high-dimensional thin polytopes: either by taking the free join with a thin polytope or by just being trivially thin. We note that trivially thin polytopes are abundant as one can simply take {\em any} Cayley polytope of at least $d+1$ lattice polytopes in $\R^d$ to get one. Hence, the following problem was posed in \cite{borger_thin_2023}.

\begin{question}\label{bkn_question}
    Suppose $d$-dimensional $P$ is spanning and thin. Is it true that if $P$ is not a free join, then $P$ is trivially thin?
    \label{question-thin}
\end{question}

This holds for $d \le 3$ \cite{borger_thin_2023}. With the help of Lawrence twists we can also easily answer\footnote{After this question was resolved in the thesis of the first author, an independent answer was also given by Selyanin in the preprint \cite{selyanin2025newtonnumbersvanishingpolytopes}.} it in dimensions $d \ge 5$.

\begin{cor}
The answer to \cref{bkn_question} is negative for any $d \ge 5$. In fact, there are infinitely-many non-isomorphic counter-examples in each dimension $d \ge 5$. 
\end{cor}

\begin{proof}

Consider the lattice pyramid $P$ over a two-dimensional lattice polygon with an interior lattice point. So,  $\dim(P)=3$, $\deg(P)=2$, and $P$ is thin. Choose any $k \in \Z_{\ge1}$. Applying \cref{lem-lawrence} yields infinitely many non-isomorphic Lawrence twists of $P$ in dimension $3+2k$ that are not free joins. As by \cref{cor-degree} they have degree at least $2+k$, they are not trivially thin. This proves the statement for $d \ge 5$ odd. 

 The result in an even dimension $d \ge 6$ follows in the same way starting from a lattice pyramid over a three-dimensional lattice polytope with an interior lattice point.
\end{proof}

We remark that we do not know the answer to \cref{question-thin} in dimension $d=4$. It is affirmative in the case of $4$-dimensional thin simplices \cite{kurylenko_thin_2024}.

\subsection{An example of a family of generalized Lawrence twists}

Of course, taking generalized Lawrence twists using \cref{Cayley-formula} allows even more freedom in getting parametrized families of high-dimensional thin polytopes by just starting with a thin polytope. As an application, let us consider thin simplices. In \cite{kurylenko_thin_2024} a complete classification of all $4$-dimensional thin simplices was given. There are 6 sporadic examples and just one infinite family parametrized by a discrete parameter $N \in 2 \mathbb{N}$. Now we can show that this family arises exactly from the construction described in \cref{thm:main-local}. Taking the vertices of the family from \cite{kurylenko_thin_2024} and performing a simple unimodular transformation one obtains thin $4$-dimensional simplices $S^{(N)}$
\[ \begin{pmatrix}
    v_0 & v_1 & v_2 & v_3 & v_4\\
\end{pmatrix} = \begin{pmatrix}
    N/2 & 0 & 0 & 0 & -1 &  \\
    0 & 0 &  1 & -1 & 1 &  \\
     0 & 0  & 0 & 0 & 2 &\\
    0 & 0 & 1 & 1 & 1 & 
\end{pmatrix}\] 
By taking $d=3$, $k=1$, and two lattice polytopes $P_1 = [0, N/2] \times \{0\} \times \{0\}$ and $P_2 = \conv((0,1,0),(0,-1,0),(-1,1,2))$, we see that $S^{(N)} \cong P_1 * P_2$. Note that $P_2$ projects along $U:= \R \times \{0\} \times \{0\}$ isomorphically onto $\proj_U(P_2) = \conv((1,0),(-1,0),(1,2))$. As the latter polygon is isomorphic to $2 \Delta_2$, we see that $P_2$ and $\proj_U(P_2)$ are both thin.  Hence, we are in the situation (1) of \cref{Cayley-formula} (where one just has to check for $I=\{2\}$), and $S^{(N)}$ is a generalized Lawrence twist of a thin polygon.

\subsection{An alternative proof of the thinness of \texorpdfstring{$B_k$}{Bk}-polytopes}

In \cite{selyanin2025newtonnumbersvanishingpolytopes} Selyanin defines in the context of Arnold's monotonicity theorem the following notion:

\begin{dfn}
A \emp{$B_k$-polytope} is the Cayley polytope $P_0 * \cdots * P_k$ for lattice polytopes in $\R^{n-k}$, where $\dim(P_1 + \cdots + P_k) < k$ and $\dim(P_0)=n-k$.
\end{dfn}

Selyanin showed in \cite[Corollary~1.7]{selyanin2025newtonnumbersvanishingpolytopes} that $B_k$-polytopes are thin. He used the theory of $\ell$-Newton numbers to prove this and asked 
in \cite[Question~1.10]{selyanin2025newtonnumbersvanishingpolytopes} whether there is a simpler way to verify that $B_k$-polytopes are thin. And indeed there is. Note that we are precisely in the situation (1) of \cref{Cayley-formula} (where his first polytope $P_0$ corresponds to our last polytope $P_{k+1}$). 
Hence, $B_k$-polytopes are generalized Lawrence twists. Now, if $k \le n-k$ one just has to apply Bernstein's criterion \cite{bernstein_number_1975} to $P_1, \ldots, P_k$ with $\dim(P_1 + \cdots + P_k) < k$ to get for the mixed volume $V=0$, so the $\ell^*$-polynomial of $P_0 * \cdots * P_{k+1}$ vanishes by \eqref{main-local}. Otherwise, if $k > n-k$, $V=0$ by the convention of \cref{Cayley-formula}, and we get again thinness by \eqref{main-local}. Moreover, $k \ge n-k$ actually implies that $P_0 * \cdots * P_k$ is even trivially thin (e.g., \cref{Cayley-formula}(2)).

\subsection{Infinitely many nearly thin polytopes}

Coming back to free joins the one construction previously known to get high-dimensional lattice polytopes with the same Hodge vector was to take the free join with a lattice polytope with Hodge vector $(1)$. This motivates the following definition:

\begin{dfn}
    A lattice polytope with Hodge vector $(1)$ is called \emp{nearly thin}. 
\end{dfn}

Note that nearly thin polytopes are necessarily odd-dimensional because of the palindromicity of the local $h^*$-polynomial. 

The standard example of a nearly thin lattice polytope is the interval $[0,2]$. More generally, for $d$ odd, $2 \Delta_d$ is nearly thin. What about other examples? Now, using generalized Lawrence twists the following result is immediate.

\begin{cor}
    In each odd dimension $d\geq 3$ there are infinitely many non-isomorphic nearly thin polytopes that are not free joins.
\end{cor}

    Note also that as the Lawrence twists  construction produces Cayley polytopes, any possible Hodge vector can be given by a lattice polytope of lattice width $1$. 

\subsection{Final questions}

Summing up our discussion one may ask the following:

\begin{question}\
    \begin{enumerate}
        \item Apart from free joins with nearly thin polytopes and generalized Lawrence twists (with mixed volume $V=1$ families) are there other possibilities to get an \emp{infinite} family of high-dimensional lattice polytopes with the same Hodge vector?
        \item Apart from being trivially thin, a free join with a thin polytope or arising from a generalized Lawrence twist are there only \emp{finitely} many thin polytopes in given dimension?
    \end{enumerate}
\end{question}

The reader is invited to check that in dimension $2$ every thin polytope is a generalized Lawrence twist except for $2 \Delta_2$. 

Finally, let us remark that in the thesis \cite{vadym_thesis} of the first author another construction, called \emp{total twist}, was presented motivated by \cite{GGFernando} that conjecturally preserves the Hodge vector. However, in contrast to generalized Lawrence twists that construction is \emp{sporadic}, in the sense that it is not always applicable and does not allow a larger degree of freedom.

\bibliographystyle{alpha}
\bibliography{references}

\end{document}